\documentclass[12pt]{article}
\usepackage{amsmath,amssymb,color}
\usepackage{xspace}

\newcommand{\Var}{{\rm{Var}_{\mathbb{C}}}}

\newcommand{\diag}{{\mbox{\rm diag}}}

\def\1{\underline{1}}
\def\R{{\mathbb R}}
\def\AA{{\mathbb A}}

\def\LLL{{\mathbb L}}
\def\Z{{\mathbb Z}}
\def\X{{\mathbb X}}
\def\Q{{\mathbb Q}}
\def\C{{\mathbb C}}

\def\phiphi{{\underline{\varphi}}}
\def\rr{{\underline{r}}}

\newtheorem{theorem}{Theorem}

\newtheorem{lemma}{Lemma}
\newtheorem{proposition}{Proposition}

\newenvironment{definition}
{\smallskip\noindent{\bf Definition\/}:}{\smallskip\par}

\newenvironment{remark}
{\smallskip\noindent{\bf Remark\/}.}{\smallskip\par}

\newenvironment{proof}
{\noindent{\bf Proof\/}.}{{ $\square$}\smallskip\par}

\title{Equivariant versions of higher order orbifold Euler characteristics 
\footnote{ 
finite group actions, orbifold Euler characteristic, Burnside ring,
complex quasi-projective varieties, wreath products, generating series.}
}

\author{S.M.~Gusein-Zade \thanks{The work of the first author
(Sections~\ref{secIntro}, \ref{secPower} and~\ref{secEqui}) was supported
by the grant 16-11-10018 of the Russian Science Foundation.
Address: Moscow State University, Faculty
of Mathematics and Mechanics, GSP-1, Moscow, 119991, Russia. E-mail:
sabir\symbol{'100}mccme.ru} \and I.~Luengo  \and
A.~Melle-Hern\'andez \thanks{The last two authors were partially
supported by the grant MTM2013-45710-C02-02-P.
Address: ICMAT (CSIC-UAM-UC3M-UCM); Complutense University of Madrid,  Dept. of Algebra, Madrid, 28040, Spain.
E-mail: iluengo\symbol{'100}mat.ucm.es, amelle\symbol{'100}mat.ucm.es}}
\date{}
\begin{document}
\def\eps{\varepsilon}

\maketitle

\begin{abstract}
There are (at least) two different approaches to define an equivariant analogue of the Euler charateristic
for a space with a finite group action. The first one defines it 
as an element of the Burnside ring  of the group. The second approach emerged from physics  and includes the  orbifold Euler
characteristic and its  higher order versions. Here we give a way to merge the two approaches together defining 
(in a certain setting)
higher order Euler characteristics with values in the Burnside ring of a group. We give Macdonald
type equations for these invariants. We also offer generalized (``motivic'') versions of these
invariants and formulate Macdonald type equations for them as well.
\end{abstract}

\section{Introduction}\label{secIntro}

Let $X$ be a topological space (good enough, say, a real subanalytic variety) with an action of
a finite group $G$. There are (at least) two different approaches to define an equivariant analogue
of the Euler charateristic for the pair $(X,G)$. The first one (\cite{TtD}) defines the
{\em equivariant Euler characteristic} $\chi^G(X)$ as an element of the Burnside ring $A(G)$
of the group $G$. The second approach emerged from physics (the string theory of orbifolds: \cite{Vafa},
\cite{Vafa2}).
The {\em orbifold Euler characteristic} $\chi^{orb}(X,G)$ is defined through the fixed point sets
of some subgroups of $G$ and is an integer. {\em Higher order (orbifold) Euler characteristics}
were introduced in \cite{AS} and \cite{BF} (also as integers). They can be defined through
the fixed point sets of collections of commuting elements in $G$. Here we give a way to merge
the two approaches together.

Through this paper we consider the Euler characteristic $\chi(\cdot)$ defined as the alternating
sum of the dimensions of the cohomology groups with compact support. This Euler characteristic
is not a homotopy invariant in the usual sense, but an invariant of the homotopy type defined
in terms of proper maps. It is an additive function on the algebra of (``good'') spaces.
For complex quasi-projective varieties this Euler characteristic coincides with the ``usual'' one,
i.~e., with the alternating sum of the dimensions of the usual cohomology groups.

There is the universal additive invariant on the algebra of complex constructible sets.
It takes values in the Grothendieck ring $K_0(\Var)$ of complex quasi-projective
varieties and can be regarded as a {\em generalized} (``motivic'') {\em Euler characteristic} $\chi_g(X)$.
There were defined generalized (``motivic'') analogues of the orbifold Euler characteristic and of
its higher order generalizations. First it was essentially made in \cite{BatDais} (for the Hodge-Deligne
polynomial) and then formulated precisely in \cite{GLMSteklov} and \cite{G&Ph}.
These {\em higher order generalized Euler characteristics} take values in a certain modification
of the Grothendieck ring $K_0(\Var)$ of complex quasi-projective varieties.

The Euler characteristic satisfies the {\em Macdonald type equation}
\begin{equation}\label{mac-simple}
1+\sum_{n=1}^{\infty}\chi(S^nX) \cdot t^n=(1-t)^{-\chi(X)}\,,
\end{equation}
where $S^nX=X^n/S_n$ is the $n$th symmetric power of $X$ (see, e.g., \cite{mac}). A Macdonald type
equation for a given invariant expresses the generating series of the values of the invariant
for the symmetric powers of a space (or for their analogues) as a series not depending on the space
(in this case $(1-t)^{-1}$) with the exponent equal to the value of the invariant for the space itself.
If the invariant takes values in a ring $R$ different from a number ring (i.e., from $\Z$, $\Q$, $\R$ or
$\C$), Macdonald type equations can be formulated in terms of a so-called power structure over the ring.
There are Macdonald type equations for the orbifold Euler characteristic and for its higher order
analogues (\cite{Wang}, \cite{T}) and also for the equivariant Euler characteristic with values in
the Burnside ring $A(G)$ of $G$ (see Lemma~\ref{lemma1} below). 
An analogue of these equations for the generalized (``motivic'') higher order Euler
characteristics was obtained in \cite{G&Ph}. It was formulated in terms of the (natural) {\em power
structure} over the Grothendieck ring $K_0(\Var)$ of complex quasi-projective varieties: \cite{GLM-MRL}.
(Its version for the orbifold Hodge-Deligne polynomial was proved in \cite{WZ}.)

Here we  define 
higher order Euler characteristics with values in the Burnside ring of a group. 
This is made in the setting when there are two commuting finite group actions. Then one of them can be
treated in the way similar to ``the orbifold aproach'' and the other in the way which leads to invariants
(say, equivariant Euler characteristic) with values in the Burnside ring of the group. This situation
can be met, for example, in the following considerations. Assume that $(f,G)$ is a pair consisting of
a quasi-homogeneous polynomial $f$ and an abelian group $G$ of its (diagonal) symmetries. Such pairs are
subjects of analysis, in particular, in the Berglund-H\"ubsch-Henningson mirror symmetry \cite{BH1},
\cite{BH2}. The classical monodromy transformation of $f$ is a map of a finite order from the
Milnor fibre of $f$ into itself. It commutes with the $G$-action. Thus one has the action of
two groups: the group $G$ and the (cyclic) group generated by the monodromy transformation.
A study of the orbifold monodromy zeta function of $f$ leads to the situation when one should apply
``the orbifold approach'' to the first action whence the other one should be treated in another way,
see, e.~g., \cite{Orbifoldzeta}.
We give Macdonald type equations for the constructed invariants. We also offer generalized (``motivic'')
versions of these invariants and formulate Macdonald type equations for them as well.

\section{Power structures over rings}\label{secPower}
As it was indicted above, a Macdonald type equation for an invariant can be formulated in terms of
a power structure over the ring of values of the invariant. Let $R$ be a commutative associative ring
with unity. A power structure over the ring $R$ gives sense to expressions of the form $(A(t))^m$,
where $A(t)=1+a_1t+a_2t^2+\ldots$ is a power series with the coefficients $a_i$ from $R$
and $m$ is an element of $R$.

\begin{definition}
 A {\em power structure} over the ring $R$ is a map $\left(1+tR[[t]]\right)\times R\to 1+tR[[t]]$
 ($(A(t),m)\mapsto\left(A(t)\right)^m$) possessing the following properties:
 \begin{enumerate}
  \item[1)] $(1+a_1t+\ldots)^m=1+ma_1t+\ldots$\,;
  \item[2)] $\left(A(t)B(t)\right)^m=\left(A(t)\right)^m\left(B(t)\right)^m$;
  \item[3)] $\left(A(t)\right)^{m+n}=\left(A(t)\right)^m\left(A(t)\right)^n$;
  \item[4)] $\left(A(t)\right)^{mn}=\left(\left(A(t)\right)^m\right)^n$.
 \end{enumerate}
\end{definition}

Let $\mathfrak{m}$ be the ideal $tR[[t]]$ in the ring $R[[t]]$.

\begin{definition}
 A power structure over the ring $R$ is {\em finitely determined} if the fact that
 $A(t)\in 1+{\mathfrak{m}}^k$ implies that $\left(A(t)\right)^m\in 1+{\mathfrak{m}}^k$.
\end{definition}

The natural power structure over the ring $\R$ of integers is defined by the standard formula
for a power of a series (see, e.~g., \cite{St}):
\begin{eqnarray*}
 &\ &(1+a_1t+a_2t^2+\ldots)^m=\\
 &=&1+\sum_{k=1}^{\infty}
 \left(\sum_{\{k_i\}:\sum ik_i=k}\frac{m(m-1)\cdots(m-\sum_i k_i +1)\cdot\prod_i a_i^{k_i}}
 {\prod_i k_i!}\right)\cdot t^k.
\end{eqnarray*}

An important example of a power structure over the Grothendieck ring of quasi-projective varieties
was introduced in \cite{GLM-MRL}.

A quasi-projective variety is the complement of a projective variety in a projective one.
The Grothendieck ring $K_0(\Var)$ of complex quasi-projective varieties is the abelian group
generated by the classes  of complex quasi-projective varieties modulo the relations:
\begin{enumerate}
 \item[1)] if $X$ and $X'$ are isomorphic, then $[X]=[X']$;
 \item[2)] if $Y\subset X$ is a Zariski closed subset, then $[X]=[Y]+[X\setminus Y]$.
\end{enumerate}
The multiplication in $K_0(\Var)$ is defined by the cartesian product. The class $\LLL=[\AA_{\C}^1]$
of the complex affine line plays a special role in a number of constructions connected with the
Grothendieck ring $K_0(\Var)$. (In \cite{Borisov} it was shown that $\LLL$ is a divisor of zero
in $K_0(\Var)$.)

A power structure over the ring $K_0(\Var)$ is defined in \cite{GLM-MRL} by the formula
\begin{eqnarray}\label{geometric}
 &\ &(1+[A_1]t+[A_2]t^2+\ldots)^{[M]}=\nonumber\\
 &=&1+\sum_{k=1}^{\infty}
 \left(\sum_{\{k_i\}:\sum ik_i=k}
 \left[\left(\left(M^{\sum_i k_i}\setminus\Delta\right)\times\prod_i A_i^{k_i}\right)\left/
 {\prod_i S_{k_i}}\right.\right]\right)\cdot t^k,\label{Power}
\end{eqnarray}
where $A_i$, $i=1, 2, \ldots$, and $M$ are quasi-projective varieties ($[A_i]$ and $[M]$ are
their classes in the ring $K_0(\Var)$), $\Delta$ is the large diagonal in $M^{\sum_i k_i}$,
i.~e., the set of (ordered) $\left(\sum_i k_i\right)$-tuples of points of $M$ with at least
two coinciding ones, the group $S_{k_i}$ of permutations on $k_i$ elements acts by the
simultaneous permutations on the components of the corresponding factor $M^{k_i}$ in
$M^{\sum_i k_i}=\prod_i M^{k_i}$ and on the components of $A_i^{k_i}$.

One can see that the coefficient at $t^k$ in the right hand side of~(\ref{Power}) has the
following interpretation. Let $I:\coprod\limits_{i=1}^{\infty}A_i\to\Z$ be the ``tautological
function'' on the disjoint union $\coprod\limits_{i=1}^{\infty}A_i$ sending $A_i$ to $i$.
The coefficient at $t^k$ in the right hand side of (\ref{Power}) is represented by the configuration
space of pairs $(K,\psi)$, where $K$ is a finite subset of $M$ and $\psi$ is a map
$K\to \coprod\limits_{i=1}^{\infty}A_i$ such that $\sum\limits_{x\in K}I(\psi(x))=k$. This interpretation
makes it much easier to prove that the equation~(\ref{Power}) really defines a power structure (i.~e.,
to verify the properties 2)-4) from the definition) and also permits to prove some formulae
for generating series of classes of Hilbert schemes of ``fat points'' (zero-dimensional subschemes)
on quasi-projective manifolds (see, e.~g., \cite{Michigan}). Less formally, see~\cite{Gorsky},
\cite{BrMor}, one can say that on the variety $M$ there live particles equipped with some natural
numbers (multiplicities, masses, charges, \dots). A particle of multiplicity $n$ has a complicated
space of internal states which is parametrized by points of a quasi-projective variety $A_n$
and the coefficient at $t^k$ is the configuration space of tuples of particles with the total
multiplicity $k$.

The notion of a power structure over a ring is closely related with the notion of a $\lambda$-ring
structure. A $\lambda$-ring structure (or a pre-$\lambda$-ring structure in a certain terminology)
is an additive-to-multiplicative homomorphism $R \to 1+tR[[t]], a\mapsto \lambda_a(t)$ such that
$\lambda_a(t)=1+at +\ldots$, where the term {\em additive-to-multiplicative} means that
$\lambda_{a+b}(t)=\lambda_a(t) \cdot \lambda_b(t)$.
A $\lambda$-ring structure over a ring defines a finitely determined power structure over it
in the following way. Any series $A(t)\in 1+tR[[t]]$ can be in  a unique way
represented as $\prod_{i=1}^{\infty}\lambda_{b_i}(t^i)$, for some $b_i\in R$.
Then one defines  $\left(A(t)\right)^m:=\prod_{i=1}^{\infty}\lambda_{m b_i}(t^i).$
On the other hand, in general, there are many $\lambda$-ring structures corresponding to one
power structure over a ring. One can show that the power strucure~(\ref{geometric}) is defined by
the $\lambda$-ring structure on the Grothendieck ring $K_0(\Var)$ given by the Kapranov zeta function 
$$
\zeta_{[M]}(t):=1+\sum_{n=1}^\infty [S^n M] \cdot t^n,
$$
where $S^n M$ is the n-th symmetric power of the variety $M$.
This follows for the following equation
$$
\zeta_{[M]}=(1-t) ^{-[M]}=(1+t+t^2+\ldots)^{[M]}.
$$


{\em Burnside ring} $A(G)$ of a finite group $G$ is the Grothendieck ring of finite $G$-sets:
see, e.g., \cite{Handbook}). As an abelian group the Burnside ring $A(G)$ is freely generated by
the classes $[G/H]$ of the quotients $G/H$ for representatives $H$ of the conjugacy classes $\mathfrak{h}$
of subgroups of $G$. The multiplication is defined by the cartesian product with the diagonal action
of $G$. There is a natural power structure over the Burnside ring $A(G)$: see, e.g., \cite{ArnoldJM}.
This power structure is defined in a way similar to that one over the ring $K_0(\Var)$. Namely, if
$A_i$, $i=1, 2, \ldots$, and $M$ are finite $G$-sets, the same equation~(\ref{Power}) defines the
series $(1+[A_1]t+[A_2]t^2+\ldots)^{[M]}$ where the action of the group $G$ on the summands is
the natural (the diagonal) one. In particular, for a finite $G$-set $X$ one has
$$
 (1-t)^{-[X]}=1+[X]t+[S^2X]t^2+[S^3X]t^3+\ldots\,,
$$
where $S^kX=X^k/S_k$ is the $k$th symmetric power of the $G$-set $X$ with the diagonal action of $G$.

Let $X$ be a $G$-space. For a point $x\in X$, let $G_x=\{g\in G:\, g\cdot x=x \}$ be the isotropy subgroup
of the point $x$. For a subgroup $H\subset G$, let $X^H=\{x\in X: Hx=x\}$ be the fixed point set of $H$
($X^H=\{x\in X: H\subset G_x\}$) and let $X^{(H)}=\{x\in X: G_x=H\}$ be the set of points with the isotropy group $H$.
Let ${\rm Conjsub\,} G$ be the set of the conjugacy classes of subgroups of $G$.
For a conjugacy class ${\mathfrak{h}}\in {\rm Conjsub\,} G$, let
$X^{\mathfrak{h}}=\{x\in X: x\in X^H \mbox{\ for a subgroup\ } H\in\mathfrak{h}\}$,
$X^{(\mathfrak{h})}=\{x\in X: G_x\in \mathfrak{h}\}$.

The  {\em equivariant Euler characteristic} of a (good enough) $G$-space $X$ is defined by
\begin{equation}\label{equiEuler}
\chi^G(X):=\sum_{\mathfrak{h}\in {\rm Conjsub\,}G}\chi(X^{(\mathfrak{h})}/G)[G/H]\in A(G)\,,
\end{equation}
where $H$ is a representative of the conjugacy class $\mathfrak{h}$ (see, e.~g., \cite{TtD}).

\section{Higher order Euler characteristics \& Macdonald type equations}\label{secHigher}
The {\em orbifold Euler characteristic}
$\chi^{orb}(X,G)$ of the $G$-space $X$ is defined, e.g., in \cite{AS}, \cite{HH}:
\begin{equation}\label{chi-orb}
 \chi^{orb}(X,G)=
\frac{1}{\vert G\vert}\sum_{{(g_0,g_1)\in G\times G:}\atop{\\g_0g_1=g_1g_0}}\chi(X^{\langle g_0,g_1\rangle})
=\sum_{[g]\in {G_*}} \chi(X^{\langle g\rangle}/C_G(g))\in \Z\,,
\end{equation}
where $G_*$ is the set of the conjugacy classes of elements of $G$, $g$ is a representative of the class ${[g]}$,
$C_G(g)=\{h\in G: h^{-1}gh=g\}$
is the centralizer of $g$, and $\langle g\rangle$ and $\langle g_0,g_1\rangle$ are the subgroups of $G$ generated
by the corresponding elements.

The {\em higher order Euler characteristics} of $(X,G)$ are defined by:
\begin{equation}\label{chi-k-orb}
 \chi^{(k)}(X,G)=
\frac{1}{\vert G\vert}\sum_{{{\bf g}\in G^{k+1}:}\atop{g_ig_j=g_jg_i}}\chi(X^{\langle {\bf g}\rangle})
=\sum_{[g]\in G_*} \chi^{(k-1)}(X^{\langle g\rangle}, C_G(g))\,,
\end{equation}
where ${\bf g}=(g_0,g_1, \ldots, g_k)$, $\langle{\bf g}\rangle$ is the subgroup 
generated by $g_0,g_1, \ldots, g_k$, and (for the second, recurrent, definition)
$\chi^{(0)}(X,G)$ is defined as the usual Euler characteristic $\chi(X/G)$ of the quotient.
The orbifold Euler characteristic $\chi^{orb}(X,G)$
is the Euler characteristic $\chi^{(1)}(X,G)$ of order $1$.

For a $G$-space $X$, the cartesian power $X^n$ carries the natural action of the wreath product
$G_n=G\wr S_n=G^n \rtimes S_n$ generated by the natural action of the symmetric group $S_n$ (permutting the factors)
and by the natural (componentwise) action of the cartesian power $G^n$.
The pair $(X_n, G_n)$ should be (or can be) considered as an analogue of the symmetric power of the pair $(X,G)$.
One has the following Macdonald type equation (see \cite[Theorem A]{T})
\begin{equation}\label{higher_generating}
\sum_{n\ge 0}\chi^{(k)}(X^n, G_n)\cdot  t^n
=\left(\prod\limits_{r_1, \ldots,r_k\geq 1}\left(1-t^{r_1r_2\cdots r_k}\right)^{r_2r_3^2\cdots r_k^{k-1}}\right)
^{-\chi^{(k)}(X, G)}\,.
\end{equation}
When $k=0$,  one gets the equation~(\ref{mac-simple}) for the quotient $X/G$.

The equation~(\ref{higher_generating}) can be interpreted in the following way. Set 
$$
\X^{(k)}_{(X,G)}(t):=\sum_{n\ge 0}\chi^{(k)}(X^n, G_n)\cdot  t^n.
$$
For the natural action of the group $G$ on $G$ (regarded as a zero-dimensional finite space)
one has $\chi^{(k)}(G, G)=1$. Therefore
$$
\X^{(k)}_{(G,G)}(t)=\left(\prod\limits_{r_1, \ldots,r_k\geq 1}\left(1-t^{r_1r_2\cdots r_k}\right)^{r_2r_3^2\cdots r_k^{k-1}}\right)
^{-1}
$$
and thus
$$
\X^{(k)}_{(X,G)}(t)=\left(\X^{(k)}_{(G,G)}(t)\right)^{\chi^{(k)}(X, G)}.
$$

Let $K_0(\Var)[\LLL^s]_{s\in\Q}$ be the modification of the Grothendieck ring $K_0(\Var)$ obtained by
adding all rational powers of $\LLL$. (This includes $\LLL^{-1}$ and thus $K_0(\Var)[\LLL^s]_{s\in\Q}$ 
contains the localization $K_0(\Var)_{(\LLL)}$ of the ring $K_0(\Var)$ by $\LLL$. Pay attention
that $\LLL$ is a zero divisor in $K_0(\Var)$ (\cite{Borisov}) and therefore the natural map
$K_0(\Var)\to K_0(\Var)_{(\LLL)}$ is not injective.
Therefore the natural map $K_0(\Var)\to K_0(\Var)[\LLL^s]_{s\in\Q}$ is not injective as well.) 
A power structure on the ring
$K_0(\Var)[\LLL^s]_{s\in\Q}$ can be defined as a sort of an extension of that on $K_0(\Var)$
using the following equations:
\begin{enumerate}
 \item[1)] $\left( A (\LLL^s t)\right)^{[M]}= \left( A (t)\right)^{[M]}|_{t\mapsto \LLL^s t}$\,;
 \item[2)] $\zeta_{\LLL^s  [M]}(t)=\zeta_{[M]}(\LLL^s t)$.
\end{enumerate}
The second equation permits to define $\left(A(t)\right)^{\LLL^s  [M]}.$
One has to represent $A(t)$ as $\prod_{i=1}^\infty \zeta_{b_i}(t^i)$, $b_i\in K_0(\Var)$, and then
to define $\left(A(t)\right)^{\LLL^s [M]}$ by
$$
\left(A(t)\right)^{\LLL^s  [M]}:=\prod_{i=1}^\infty \zeta_{b_i[M]}(\LLL^s t^i).
$$

\begin{remark}
 If, in the considerations bellow, one uses only non-negative weights $\varphi_i$, one can work
 with the ring $K_0(\Var)[\LLL^s]_{s\in\Q_{\geq 0}}$. The natural map 
 $K_0(\Var)\to K_0(\Var)[\LLL^s]_{s\in\Q_{\geq 0}}$ is injective. 
\end{remark}

Now let $X$ be a smooth quasi-projective variety of dimension $d$ with an (algebraic) action of
the group $G$. To define the  {\em higher order generalized (orbifold) Euler characteristics}
of the pair $(X,G)$, one has to use the so called age (or fermion shift) $F_x^g$ of an element
$g\in G$ at a fixed point $x$ of $g$ defined in \cite{Zaslow}, \cite{Ito-Reid}.
The element $g$ acts on  the tangent space $T_xX$ as an automorphism of finite order.  This action
on $T_xX$ can be represented by a diagonal matrix
$\diag(\exp(2\pi i \theta_1), \ldots, \exp(2\pi i \theta_d))$
with $0\le\theta_j<1$ for $j=1,2, \ldots, d$ ($\theta_j$ are rational numbers). The {\em age}
of the element $g$ at the point $x$
is defined by $F^{g}_{x}=\sum_{j=1}^{d}\theta_j\in\Q_{\ge 0}$. 

For $g\in G$, let the number of $C_G(g)$-orbits in the the set of connected components of the
fixed point set  $X^{\langle g\rangle}$ be equal to $N_g$, and let
$X^{\langle g\rangle}_1$, $X^{\langle g\rangle}_2$, \ldots, $X^{\langle g\rangle}_{N_g}$
be the unions of the components of each of the orbits. 
For $1\leq \alpha_g \leq N_g$, let $F^{g}_{\alpha_g}$ be the age $F^{g}_{x}$ at a point of
$X^{\langle g\rangle}_{\alpha_g}$ (this age does not depend on the point $x\in X^{\langle g\rangle}_{\alpha_g}$).

For a rational number $\varphi_1\in \Q$, the generalized orbifold Euler characteristic
of weight $\varphi_1$ of the pair $(X,G)$ is defined by
\begin{equation}\label{generalized_orbifold}
[X,G]_{\varphi_1}:=\sum_{[g]\in G_*}\sum_{\alpha_g=1}^{N_g}
[X^{\langle g\rangle}_{\alpha_g}/C_G(g)]\cdot\,\LLL^{\varphi_1 F^{\langle g\rangle}_{\alpha_g}}\in 
K_0(\Var)[\LLL^{s}]_{s\in \Q}\,.
\end{equation}

For $\varphi_1=1$ one gets the definition of the generalized orbifold Euler characteristic from \cite{GLMSteklov}
inspired by the definition of the orbifold Hodge-Deligne polynomial from \cite{BatDais}. (This generalized
orbifold Euler characteristic maps to the orbifold Hodge-Deligne polynomial by the natural ring homomorphism
$e:K_0(\Var)\to\Z[u,v]$.) For $\varphi_1=0$ one gets the so called inertia stack class: see, e.g., \cite{FLNU}.

For a subgroup $H\subset G$, and for  an   $H$-invariant  submanifold $Y\subset X$, let us define  $[Y,H]_{X,\varphi_1}$ by
\begin{equation}\label{generalized_orbifold_inX}
[Y,H]_{X,\varphi_1}:=\sum_{[g]\in H_*}\sum_{\alpha_g=1}^{N_g}
[Y^{\langle g\rangle}_{\alpha_g}/C_H(g)]\cdot\,\LLL^{\varphi_1 F^{\langle g\rangle}_{\alpha_g}}\in 
K_0(\Var)[\LLL^{s}]_{s\in \Q}\,,
\end{equation}
where $Y^{\langle g\rangle}_{\alpha_g}=X^{\langle g\rangle}_{\alpha_g}\cap Y$,   $F^{\langle g\rangle}_{\alpha_g}$ is the age of $g$ at a point  of  
$X^{\langle g\rangle}_{\alpha_g}$. (Pay attention  that the age of an element $g$ is determined from its action on $T_{x} X$.)

Let  $\underline{\varphi}=(\varphi_1,\varphi_2,\ldots)$ be a fixed sequence of rational numbers.
The {\em generalized (orbifold) Euler characteristics of order} $k$ of weight 
$\underline{\varphi}$ of the pair $(X,G)$ 
is defined recursively by
\begin{equation}\label{higher_generalized_orbifold}
 [X,G]^k_{\underline{\varphi}}:=\sum_{[g]\in G_*}
\sum_{\alpha_g=1}^{N_g}[X^{\langle g\rangle }_{\alpha_g}, C_G(g)]^{k-1}_{X,\underline{\varphi}}
\cdot\,\LLL^{\varphi_k F^{\langle g\rangle }_{\alpha_g}}\in 
K_0(\Var)[\LLL^{s}]_{s\in \Q}\,,
\end{equation}
where, for an $H$-invariant submanifold $Y\subset X$ ($H\subset G$),   $[Y,H]^1_{X,\underline{\varphi}}:=[Y,H]_{X,\varphi_1}$ is defined by  (\ref{generalized_orbifold_inX}).
(Alternatively one can start from $k=0$ using the definition $[Y,H]^0_{X,\underline{\varphi}}=[Y,H]^0=[Y/H]$.)

One has the following Macdonald type equations (\cite{G&Ph})
\begin{equation}
\sum_{n\ge 0}[X^n, G_n]_{\phiphi}^{k}\cdot t^n
=\left(\prod\limits_{r_1, \ldots,r_k\geq 1}\left(1-\LLL^{\Phi_k(\rr)d/2}\cdot t^{r_1r_2\cdots r_k}\right)^{r_2r_3^2\cdots r_k^{k-1}}\right)
^{-[X, G]_{\phiphi}^{k}}\,,
\end{equation}
where 
$$\Phi_k(r_1,\ldots,r_k)=\varphi_1(r_1-1)+\varphi_2r_1(r_2-1)+\ldots+
\varphi_kr_1r_2\cdots r_{k-1}(r_k-1).$$

\section{Equivariant higher order Euler characteristics}\label{secEqui}

Assume that $X$ is a (good enough) topological space with
\emph{commuting} actions of two finite groups $G_{O}$ and $G_{B}$
(or equivalently with an action of the product $G_{O}\times G_{B}$).
The quotient $X/G_O$ carries the natural $G_B$-action 
and thus one can define 
$\chi^{(0,G_B)}(X; G_O,G_B)$ as $\chi^{G_B}(X/G_O)\in A(G_{B})$.

For an element $g\in G_O$,  the fixed point set  $X^{\langle g\rangle}$
is $G_B$-invariant and the quotient $X^{\langle g\rangle}/C_{G_O}(g)$
by the centralizer $C_{G_O}(g)$ carries the natural $G_B$-action. 

\begin{definition}
 The \emph{equivariant orbifold  Euler characteristics} of $(X; G_O,G_B)$ is 
\begin{eqnarray}
\hskip-10pt\chi^{(1,G_B)}(X; G_O,G_B)&:=&\sum_{[g]\in {G_{O*}}} \chi^{G_B}(X^{\langle g\rangle}/C_{G_O}(g))
\nonumber\\
&{\ =}&\sum_{[g]\in G_{O*}} \chi^{(0,G_B)}(X^{\langle g\rangle}; C_{G_O}(g),G_B)\in A(G_B)\,.
\end{eqnarray}
\end{definition}

The  equivariant higher order Euler characteristics
are defined recursively in the same way as the non-equivariant one.

\begin{definition}
 The \emph{equivariant Euler characteristics of order} $k$ of $(X; G_O,G_B)$ is 
\begin{equation}\label{discrite-equiv}
\chi^{(k,G_B)}(X; G_O,G_B):=
\sum_{[g]\in {G_O}_*} \chi^{(k-1, G_B)}(X^{\langle g\rangle}; C_{G_O}(g), G_B)\in A(G_B).
\end{equation}
\end{definition}

Definition (\ref{discrite-equiv}) is convinient for the proof of Theorem~\ref{discrite-ver}.
Iterating (\ref{discrite-equiv}) one gets the following equation for the equivariant
higher order Euler characteristics:
\begin{equation} \label{part-2}
  \chi^{(k,G_B)}(X; G_O,G_B)=
\sum_{[\phi] \in {\rm Hom\,}(\Z^k, G_O)/G_O}  \chi^{G_B} (X^{\langle \phi \rangle}/ C_{G_O}(\phi))\,, 
\end{equation}
where the group $G_O$ acts on ${\rm Hom\,}(\Z^k, G_O) $
by conjugation, $X^{\langle \phi \rangle}$ is the fixed point set of the image of $\phi$,  
$C_{G_O}(\phi)=\{g\in G_O: g^{-1}\phi g=\phi \}$.

\begin{theorem}\label{discrite-ver}
One has
\begin{equation}\label{eq_theo1}
\sum_{n\ge 0} \chi^{(k,G_B)}(X^n; (G_O)_{n},G_B)\cdot t^n
=\left(\prod\limits_{r_1, \ldots,r_k\geq 1}\left(1-t^{r_1r_2\cdots r_k}\right)^{r_2r_3^2\cdots r_k^{k-1}}\right)
^{-\chi^{(k,G_B)}(X; G_O,G_B)}\,,
\end{equation}
where the exponent in the right hand side is defined by the power structure over the Burside ring $A(G_B)$.
\end{theorem}

\begin{proof}
 The proof essentially repeats, e.g., the one in \cite{T} (see also\cite{G&Ph}). One has to pay
attention to two facts. First, all subsets participating in the course of the proof in \cite{T}
are $G$-invariant. Second, one has to use the Macdonald type equation for the equivariant Euler
characteristic $\chi^G(\cdot)$: see (\ref{MacEqui}) below. 
\end{proof}

The Macdonald type equation for the equivariant Euler characteristic must be known.
However, we have not found it in the literature. Therefore we put its proof here.

\begin{lemma}\label{lemma1}
 For a $G$-space $X$ one has
 \begin{equation}\label{MacEqui}
  1+ 
  \sum_{n=1}^{\infty} \chi^G(S^nX) \cdot t^n
  = (1-t)^{-\chi^G(X)}\in 1+tA(G)[[t]]\,,
 \end{equation}
where the right hand side is defined by the power structure over the Burnside ring $A(G)$
(or by the $\lambda$-structure on it).
\end{lemma}

\begin{remark}
 Note that, for a finite $G$-set $X$, the equation~(\ref{MacEqui}) is just the definition of its
 right hand side. (In this case $\chi^G(X)=[X]\in A(G)$.)
\end{remark}

\begin{proof}
 Let us denote the left hand side of (\ref{MacEqui}) by $\chi^G\zeta_{X,G}(t)$.
 If $X=X_1\coprod X_2$ is a decomposition of $X$ into two $G$-subspaces, one has
 $$
 \chi^G\zeta_{X,G}(t)=\chi^G\zeta_{X_1,G}(t)\cdot\chi^G\zeta_{X_2,G}(t)\,.
 $$
 (This follows from the identities $S^nX=\coprod\limits_{m=0}^n (S^{m}X_1)\times(S^{n-m}X_2)$,
 $\chi^G(X\times Y)=\chi^G(X)\chi^G(Y)$.) Therefore it is sufficient to prove (\ref{MacEqui}) for the elements
 of a decomposition of $X$ into $G$-invariant subspaces. A ``good enough'' $G$-space (say, a real subanalytic
 one) can be represented as the disjoint union of subspaces of the form $\sigma^d\times (G/H)$, where
 $H$ is a subgroup of $G$, $G/H$ is the corresponding $G$-set (the quotient), and $\sigma^d$ is the open
 cell of dimension $d$ with the trivial $G$-action. Let $\overline{\sigma}^d\supset\sigma^d$ be the closed
 $d$-dimensional ball. Since $\overline{\sigma}^d$ can be $G$-equivariantly contructed to a point,
 $S^k(\overline{\sigma}^d\times(G/H))$ can be contracted to $S^k(G/H)$. Therefore
 $$
 \chi^G(S^k(\overline{\sigma}^d\times(G/H)))=\chi^G(S^k(G/H))=[S^k(G/H)]
 $$
 (see, e.g., \cite{TtD} where this is formulated for finite $G$-$CW$-complexes)
 and thus
 $$
 \chi^G\zeta_{S^k(\overline{\sigma}^d\times(G/H))}(t)=1+\sum_{i=1}^{\infty}[S^i(G/H)]t^i
 =
 (1-t)^{-[G/H]}=(1-t)^{-\chi^G(\overline{\sigma}^d\times(G/H))}\,.
 $$
 The equation~(\ref{MacEqui}) obviously holds for $X=\sigma^d\times (G/H)$ with $d=0$
 (when $\sigma^d$ is a point).
 Assume that it holds for $X=\sigma^d\times (G/H)$ with $d<d_0$. One has a decomposition
 $\sigma^{d_0}=\sigma^{d_0}\coprod\sigma^{d_0}\coprod\sigma^{d_0-1}$. Therefore
 $$
 \chi^G\zeta_{\sigma^{d_0}\times (G/H)}(t)=\left(\chi^G\zeta_{\sigma^{d_0}\times (G/H)}(t)\right)^2\cdot
 \chi^G\zeta_{\sigma^{(d_0-1)}\times (G/H)}(t)\,,
 $$
 \begin{eqnarray*}
 \chi^G\zeta_{\sigma^{d_0}\times (G/H)}(t)&=&
 \left(\chi^G\zeta_{\sigma^{(d_0-1)}\times (G/H)}(t)\right)^{-1}=
 \left((1-t)^{-(-1)^{(d_0-1)}[G/H]}\right)^{-1}\\
 &=&(1-t)^{-(-1)^{d_0}[G/H]}=(1-t)^{-\chi^G(\sigma^{d_0}\times(G/H))}\,.
 \end{eqnarray*}
\end{proof}

The equation~(\ref{eq_theo1}) has an interpretation similar to the one for ``usual''
(non-equivariant) higher order Euler characteristics. Set 
$$
\X^{(k)}_{(X;G_O,G_B)}(t):=\sum_{n\ge 0}\chi^{(k,G_B)}(X^n; {G_O}_n, G_B)\cdot  t^n.
$$
Let $G_O$ (regarded as a zero-dimensional finite space) be endowed with the natural action of
the group $G_O$ and with the trivial action of the group $G_B$.
Then one has $\chi^{(k,G_B)}(G_O; G_O, G_B)=1$. Therefore
$$
\X^{(k)}_{(G_O;G_O,G_B)}(t)=
\left(\prod\limits_{r_1, \ldots,r_k\geq 1}\left(1-t^{r_1r_2\cdots r_k}\right)^{r_2r_3^2\cdots r_k^{k-1}}
\right)^{-1}
$$
and thus
$$
\X^{(k)}_{(X;G_O,G_B)}(t)=\left(\X^{(k)}_{(G_O;G_O,G_B)}(t)\right)^{\chi^{(k,G_B)}(X; G_O,G_B)}.
$$

\section{Equivariant generalized higher order Euler characteristics}\label{secGeneral}
For a finite group $G$,
let $K_0^{G}(\Var)$ be the Grothendieck ring of complex quasi-projective $G$-varieties. 
By that we mean the free abelian group generated by the $G$-isomorphism classes $[X,G]$
(or $[X]$ for short) of complex quasi-projective varieties
$X$ with $G$-actions modulo the relation:
$[X,G]=[Y,G]+[X\setminus Y,G]$ for a Zariski closed $G$-invariant subvariety $Y$ of $X$.
The multiplication in $K_0^{G}(\Var)$ is defined by the cartesian product with the diagonal $G$-action.
Let $\LLL\in K_0^{G}(\Var)$ be the class of the affine line $\AA_{\C}^1$ with the trivial $G$-action.

\begin{remark}
Usually, in the definition of the Grothendieck ring of complex quasi-projective $G$-varieties,
one adds one more relation:
if $E\to X$ is a $G$-equivariant vector bundle of rank $n$, then $[E]=[\AA_{\C}^n\times X]$.
We do not need this relation for the construction. One can say that we use the Grothendieck ring
denoted by $K_0^{', G}(\Var)$ in \cite{bittner}. The same definition was used in \cite{Mazur}.
An equation which holds in the equivariant Grothendieck ring $K_0^{G}(\Var)$ defined here,
holds in the ``traditional'' one as well.
\end{remark}

There is a natural power structure over the (equivariant) Grothendieck ring $K_0^{G}(\Var)$. Its geometric definition
is given in the same way as the usual power structure over the (non-equivariant) Grothendieck ring $K_0(\Var)$
in \cite{GLM-MRL}: for complex quasi-projective $G$-varieties $A_i$, $i=1, 2,\ldots$, and $M$ one has
\begin{eqnarray}\label{power}
&{}&(1+[A_1] t+[A_2] t^2+\ldots)^{[M]}
\nonumber\\
&{=}&1+\sum_{k=1}^{\infty}\sum_{\{k_i\}:\sum ik_i=k}
\left[\left(
(\prod_i M^{k_i}
)
\setminus\Delta
\right)
\times\prod_i A_i^{k_i}/\prod_iS_{k_i}\right] t^k\,,
\end{eqnarray}
where $\Delta$ is the ``big diagonal'' in $M^{\sum k_i}$, the symmetric groups $S_{k_i}$
act by the simultaneous permutations of the components of the corresponding factor $M^{k_i}$
in $M^{\sum k_i}=\prod_i M^{k_i}$
and on the components of $A_i^{k_i}$.
One has to take into account that all summands in the right hand side of the equation~(\ref{power})
are $G$-invariant  spaces.  The proof of the necessary properties of the power structure
is the same as in \cite{GLM-MRL}. This power structure is induced by the $\lambda$-structure
on  $K_0^{G}(\Var)$ defined by the Kapranov zeta-function.
For  a quasi-projective $G$-variety $X$, the series $(1-t)^{-[X]}$
is the 
equivariant Kapranov zeta-function of $X$:
$\zeta_{[X]}(t):=1+[X]\cdot t+[S^2 X]\cdot t^2+ [S^3X]\cdot t^3+\ldots=(1-t)^{-[X]},$
where $S^k X=X^k/S_k$ is the $k$-th symmetric power of the $G$-variety $X$ with the natural $G$-action 
(see, e.g., \cite{Kap}, \cite{Mazur} for the non-equivariant case). The map $\chi^G:K_0^{G}(\Var)\to A(G)$
is a $\lambda$-ring homomorphism.

In what follows we need the following statement.

\begin{lemma}\label{lemma2}
Let  $p:E\to X$ be a $G$-equivariant vector bundle of rank $n$ such that for
each $x\in X$ the action of the isotropy subgroup $G_{x}$ on the  fibre $E_x=p^{-1}(x)$ is trivial.  
Then $[E]=\LLL^n [X]$.  
\end{lemma}

\begin{proof}
Factorizing by the action of $G$ one gets the map $\check{p}: E/{G} \to X/G$
which is a vector bundle (due to the triviality of the action of $G_{x}$ on $E_x$ ).
According to \cite{Serre} the quotient $X/G$ can be covered by Zariski open subsets $U_i$ such that
over each $U_i$ the fibre bunble $\check{p}$ is trivial.
Therefore 
\begin{equation}\label{trivialization}
 \check{p}^{-1}(U_i) \cong  U_i\times  \AA_{\C}^n.
\end{equation}
If $V_i= {\pi}^{-1}(U_i)$, where $\pi:X\to X/G$ is the canonical factorization map,
then the trivialization $(\ref{trivialization})$
gives a trivialization of the bundle $p$ over $V_i$, i.e. an isomorphism between $p^{-1}(V_i)$ and
$V_i\times  \AA_{\C}^n$ with the trivial
$G$-action on  $\AA_{\C}^n$. This gives the statement.
\end{proof}

In what follows we need the following properties of the power structure over
the equivariant Grothendieck ring $K_0^{G}(\Var)$.

\begin{proposition}\label{prop1}
Let $A_i$ and $M$ be  $G$-varieties, and let $A(t):=1+[A_1]t+[A_2]t^2+\ldots\in K_0^{G}(\Var)$.
Then, for $s\ge0$,
\begin{equation}\label{Lt}
 \left(A(\LLL^s t)\right)^{[M]} = \left(A(t)\right)^{[M]}|_{t \mapsto \LLL^s t}\,.
\end{equation}
\end{proposition}

\begin{proof}
The coefficient at the monomial $t^k$ in the power series $\left(A(t)\right)^{[M]}$ is a sum
of the classes of the varieties of the form
$$
V=\left(
(\prod_i M^{k_i}
)
\setminus\Delta
\right)
\times\prod_i A_i^{k_i}/\prod_iS_{k_i}
$$
with $\sum ik_i=k.$ The corresponding summand $\widetilde V$ in the coefficient
at the monomial $t^k$ in the power series $\left(A(\LLL^s t)\right)^M$
has the form
$$\widetilde V=\left(
(\prod_i M^{k_i}
)
\setminus\Delta
\right)
\times\prod_i (\LLL^{si}A_i)^{k_i}/\prod_iS_{k_i}.
$$
There is the natural map $\widetilde V \to V$
which is a $G$-equivariant vector bundle of rank $sk$ satisfying the conditions of Lemma~\ref{lemma2}.
By (\ref{lemma2}) one has $[\widetilde{V}]=\LLL^{sk}[V]$, what implies (\ref{Lt}).
\end{proof}

Proposition~\ref{prop1} together with the multiplicative properrty
of the power structure ($\left(A(t)\right)^{[M][N]}=(A(t)^{[M]})^{[N]}$) implies the following statement.

\begin{proposition}\label{prop2}
For a complex quasi-projective $G$-variety $X$ one has 
\begin{equation}\label{zetaLt}
 \zeta_{\LLL[X]}(t) =\zeta_{[X]}(\LLL t)\,.
\end{equation}
\end{proposition}

Here we suggest an equivariant version of the generalized higher order Euler characteristic with values in
the modification $K_0(\Var)[\LLL^{s}]\,_{s\in \Q}$ of the equivariant Grothendieck ring of complex quasi-projective varieties.

Let $X$ be a smooth quasi-projective variety of dimension $d$ with
commuting (algebraic) actions of two finite groups $G_{O}$ and $G_{B}$
(or equivalently with an action of the product $G_{O}\times G_{B}$).

Let us define the \emph{zero order equivariant generalized Euler characteristic}  of $(X;G_O,G_B)$ by
\begin{equation}\label{equiv_gen_0}
 [X; G_O,G_B]^{0,G_B}:=[X/G_O]\in K_0^{G_B}(\Var).
\end{equation}

For $g\in G_O$, let the set of the connected components of the fixed point set  $X^{\langle g\rangle}$
consist of $N_g$ $(C_{G_O}(g)\times G_B)$-orbits and let
$X^{\langle g\rangle}_1$, $X^{\langle g\rangle}_2$, \ldots, $X^{\langle g\rangle}_{N_g}$
be the unions of the components of each of the orbits. 
For $1\leq \alpha_g \leq N_g$ let $F^{g}_{\alpha_g}$ be the age $F^{g}_{x}$ at a point of
$X^{\langle g\rangle}_{\alpha_g}$ (this age does not depend on the point
$x\in X^{\langle g\rangle}_{\alpha_g}$).

\begin{definition}
For a rational number $\varphi_1\in \Q$, the \emph{generalized orbifold Euler characteristic
of weight} $\varphi_1$ of $(X; G_O,G_B)$  is defined by  
\begin{eqnarray}\label{equiv_gen_orb}
&{\ }&[X;G_O,G_B]_{\varphi_1}^{1,G_B}:=\sum_{[g]\in (G_O)_*}\sum_{\alpha_g=1}^{N_g}
[X^{\langle g\rangle}_{\alpha_g}/C_{G_O}(g)]\cdot\,\LLL^{\varphi_1 F^{\langle g\rangle}_{\alpha_g}}
\nonumber\\
&{=}&\sum_{[g]\in (G_O)_*}\sum_{\alpha_g=1}^{N_g}
[X;G_O,G_B]^{0,G_B}
\cdot\,\LLL^{\varphi_1 F^{\langle g\rangle}_{\alpha_g}}
\in K_0^{G_B}(\Var)[\LLL^{s}]\,_{s\in \Q}\,.
\end{eqnarray}
 \end{definition}

For a subgroup $H_O\subset G_O$ and for an $H_O\times G_B$-invariant submanifold $Y\subset X$, let us define $[Y;H_O,G_B]_{X,\varphi_1}^{1,G_B}$ by
\begin{equation}\label{equiv_gen_orb-inX}
[Y;H_O,G_B]_{X,\varphi_1}^{1,G_B}:=\sum_{[g]\in (H_O)_*}\sum_{\alpha_g=1}^{N_g}
[Y^{\langle g\rangle}_{\alpha_g}/C_{H_O}(g)]\cdot\,\LLL^{\varphi_1 F^{\langle g\rangle}_{\alpha_g}}
\in K_0^{G_B}(\Var)[\LLL^{s}]\,_{s\in \Q}\,,
\end{equation}
 where $Y^{\langle g\rangle}_{\alpha_g}=X^{\langle g\rangle}_{\alpha_g}\cap Y$, $F^{\langle g\rangle}_{\alpha_g}$ is the age of $g$
at a point of $X^{\langle g\rangle}_{\alpha_g}$. 

Let  $\underline{\varphi}=(\varphi_1,\varphi_2,\ldots)$ be a fixed sequence of rational numbers.

\begin{definition}
The {\em equivariant generalized Euler characteristics of order} $k$ {\em of weight} 
$\underline{\varphi}$ of $(X; G_O,G_B)$, as an element of $K_0^{G_B}(\Var)[\LLL^{s}]\,_{s\in \Q}$, is defined recursively by
\begin{equation*}
 [X; G_O,G_B]^{k, G_B}_{\underline{\varphi}}:=\sum_{[g]\in (G_O)_*}
\sum_{\alpha_g=1}^{N_g}[X^{\langle g\rangle }_{\alpha_g}; C_G(g), G_B]^{(k-1), G_B}_{X,\underline{\varphi}}
\cdot\,\LLL^{\varphi_k F^{\langle g\rangle }_{\alpha_g}},
\end{equation*}
where, for an $H_0\times G_B$-invariant submanifold $Y\subset X$ ($H_O\subset G_O$) ,
$[Y;H_O,G_B]^{0, G_B}_{X,\underline{\varphi}}=[Y;H_O,G_B]^{0, G_B}$ is the 
zero order equivariant generalized 
Euler characteristic given by (\ref{equiv_gen_0}).
\end{definition}

\begin{remark}
Iterating this definition one can write an equation for $[X; G_O,G_B]^{k, G_B}_{\underline{\varphi}}$
analogues to (\ref{part-2}). For a homomorphism $\phi:\Z^k \to G_O$,   
let the set of the connected components of the fixed point set  $X^{\langle \phi \rangle}$
consist of $N_\phi$ $(C_{G_O}(g)\times G_B)$-orbits and let
$X^{\langle \phi \rangle}_1$, $X^{\langle \phi \rangle}_2$, \ldots, $X^{\langle \phi \rangle}_{N_{\phi}}$
be the unions of the components of each of the orbits. For $x\in X^{\langle \phi \rangle}$,
let the shift $F_x^{\phi}$ be defined as $\sum_{i=1}^k \varphi_i F_x^{\phi(e_i)}$, where
$e_1,\ldots, e_k$ is the standard basis of $\Z^k$. 
For $1\leq \alpha_\phi \leq N_g$ let $F^{\phi}_{\alpha_\phi}$ be the  shift $F^{\phi}_{x}$ at a point of
$X^{\langle \phi\rangle}_{\alpha_\phi}$. One has:
$$[X; G_O,G_B]^{k, G_B}_{\underline{\varphi}}=
\sum_{[\phi] \in {\rm Hom\,}(\Z^k, G_O)/G_O}  \sum_{\alpha_\phi=1}^{N_\phi}  
[X^{\langle \phi \rangle}/ C_{G_O}(\phi)] \cdot  \LLL^{F^{\phi}_{\alpha_\phi}}\, .
$$
\end{remark}

\begin{theorem}\label{main-equiv}
Let $X$ be a (smooth) quasi-projective variety of dimension $d$ with
commuting actions of two finite groups $G_{O}$ and $G_{B}$. Then
\begin{equation*}
\sum_{n\ge 0}[X^n; (G_O)_n, G_B]_{\phiphi}^{k, G_B}\cdot t^n
=\left(\prod\limits_{r_1, \ldots,r_k\geq 1}\left(1-\LLL^{\Phi_k(\rr)d/2}\cdot t^{r_1r_2\cdots r_k}\right)^{r_2r_3^2\cdots r_k^{k-1}}\right)
^{-[X; G_0,G_B]_{\phiphi}^{k}}\,,
\end{equation*}
where 
$$\Phi_k(r_1,\ldots,r_k)=\varphi_1(r_1-1)+\varphi_2r_1(r_2-1)+\ldots+
\varphi_kr_1r_2\cdots r_{k-1}(r_k-1).$$
\end{theorem}

\begin{proof}
The proof in \cite{G&Ph} was by induction started from $k=1,$ i.e. from the generalized orbifold case.
The latter one was treated in \cite{GLMSteklov}. One can easily see that the both proofs admit an action
of an additional group $G_B$, i.e. all the subspaces are $G_B$-invariants. In particular, symmetric products
participating in the proof of \cite{GLMSteklov} carry the natural action of $G_B$.
Using Propositions~\ref{prop1} and \ref{prop2}  
 the corresponding generating series can be written as exponents in terms of the power structure
over the modification  $K_0^{G_B}(\Var)[\LLL^s]_{s\in \Q}$ of the equivariant Grothendieck ring
of quasi-projective varieties.
\end{proof}

\end{document}